\def\tr{\mathrm{tr\,}}

\def\p{{\rm p}}
\def\I{{\rm I}}

\def\d{{\rm d}}

\def\v{{\bf v}}
\def\w{{\bf w}}

\def\vstrut{\phantom{\biggm|}}

\def\diag#1{{{\rm diag}(#1)}}

\documentclass[10pt,reqno]{amsart}
     \makeatletter
     \def\section{\@startsection{section}{1}%
     \z@{.7\linespacing\@plus\linespacing}{.5\linespacing}%
     {\bfseries
     \centering
     }}
     \def\@secnumfont{\bfseries}
     \makeatother
\newtheorem{theorem}{Theorem}[section]
\newtheorem{lemma}[theorem]{Lemma}
\newtheorem{proposition}[theorem]{Proposition}
\newtheorem{corollary}[theorem]{Corollary}
\theoremstyle{definition}
\newtheorem{definition}[theorem]{Definition}
\newtheorem{example}[theorem]{Example}
\theoremstyle{remark}
\newtheorem{remark}[theorem]{Remark}
\numberwithin{equation}{section}
\usepackage{enumerate}
\begin{document}

\title{Krawtchouk-Griffiths Systems I: Matrix Approach}
\author{Philip Feinsilver}
\address{Department of Mathematics\\
Southern Illinois University, Carbondale, Illinois 62901, U.S.A.}
\email{phfeins@siu.edu}


 \begin{abstract}
We call \textsl{Krawtchouk-Griffiths} systems, or KG-systems, systems of multivariate polynomials orthogonal with respect
to corresponding multinomial distributions. The original Krawtchouk polynomials	are orthogonal with respect to a binomial distribution.
Our approach is to work directly with matrices comprising the values of the polynomials at points of a discrete grid based on the possible
counting values of the underlying multinomial distribution. The starting point for the construction of a KG-system is a \textsl{generating matrix} satisfying the
K-condition, orthogonality with respect to the basic probability distribution associated to an individual step of the multinomial process.
The variables of the polynomials corresponding to matrices may be interpreted as quantum observables in the real case, or
quantum variables in the complex case. The structure of the recurrence relations for the orthogonal polynomials	 is presented
with multiplication operators as the matrices corresponding to the quantum variables. 
An interesting feature is that the associated random walks correspond
to the Lie algebra of the representation of symmetric tensor powers of matrices.
\end{abstract}

\maketitle

\begin{section}{Introduction}
The original paper of Krawtchouk \cite{Kra} presents polynomials orthogonal with respect to a general
binomial distribution and discusses the connection with Hermite polynomials. Group-theoretic approaches
were presented by Koornwinder \cite{K}. We developed an approach
\cite{FK} using matrices, observed as well in the work of N.\, Bose \cite{B}, with entries the values of Krawtchouk polynomials
on their domain. Matrix constructions for the multivariate case basic to this work were presented in \cite{FS2}, acknowledging
the foundational work of Griffiths, see \cite{DG,G1,G2,G3}, also \cite{MT}. Another point of view, the analytic and quantization aspects,
is that of ``Bernoulli systems", based on the approach of \cite{FS1}. See \cite{GR1,GR2} for probabilistic aspects as well.\\

The symmetric representation, the action of operators on symmetric tensor space,  is a main part of invariant theory. See \cite{Sl} in the context of
coding theory and for the invariant theory setting \cite{Sp}, e.g. The infinite dimensional
version,  boson Fock space, is fundamental in the second quantization of operators in quantum theory and  is a major component
of quantum probability theory.\\

Important aspects of this work include a general recurrence formula for symmetric representations in Section 2. 
The Transpose Lemma plays an essential r\^ole in this work. The vector field formulation and Lie algebra maps are
standard constructions, newly appearing in this context, important to our approach. 
The Columns Theorem and quantum variables comprise Sections 3 and 4.
The term ``quantum variable" is used to signify unitary equivalence to a diagonal matrix, but the spectrum may
not be entirely real. Multivariate Krawtchouk polynomials make their appearance in Section 5 
with quantum variables providing recurrence formulas and linearization formulas. We conclude after a section with
important examples, especially properties involving reflections.

\begin{subsection}{Basic notations and conventions}

\begin{enumerate}[(\tt i)]

\item We consider polynomials in $d+1$ commuting
variables over a field, which we typically take to be
$\mathbf{C}$, while the constructions are valid over $\mathbf{R}$ as well.\\

\item Multi-index notations for powers. With
$n=(n_0,\ldots,n_d)$, $x=(x_0,\ldots,x_d)$:
$$ x^n = x_0^{n_0}\cdots x_d^{n_d} $$
and the total degree $|n|=\sum_i n_i$.
Typically $m$ and $n$ will denote multi-indices, with
$i,j,k,l$ for single indices. Running indices
may be used as either type, determined from the context.\\

\item Vectors, conventionally taken to be column vectors, will be indicated by boldface, with components as usual, e.g.,
$$\v=\begin{pmatrix} v_0\\v_1\\ \vdots \\ v_d \end{pmatrix} \ .$$

\item We use the following \textsl{summation convention} \\[.1in]
 repeated Greek indices, e.g., $\lambda$ or $\mu$, are summed from $0$ to $d$. \\[.1in]
 Latin indices $i$, $j$, $k$, run from $0$ to $d$ and are summed only when explicitly indicated.\\

\item For simplicity, we will always denote identity matrices of
the appropriate dimension by $I$. \\

The transpose of a matrix $A$ is denoted $A^\top$, with Hermitian transpose
$A^*$. The inner product $\langle \v,\w \rangle=\v^*\w$.
The dot product $\v \cdot \w=\sum v_iw_i=v_\lambda w_\lambda$.\\

We will use the notation $U$ to denote an arbitrary unitary matrix, orthogonal in the real case,
i.e., it is a unitary map with respect to the standard inner product.\\

\item Given $N\ge 0$, $B$ is defined as
the multi-indexed matrix having as its only non-zero entries 
$$B_{mm}= \binom{N}{m}=\frac{N!}{m_0!\ldots m_d!}$$
the multinomial coefficients of order $N$. \\

\item For a tuple of numbers, $\diag{\ldots}$ is the diagonal matrix with the tuple providing the entries forming the main diagonal.\\

\item The bar notation, $A\rightarrow \bar A$, is only used for the symmetric power mapping, 
with complex conjugation indicated by (included in) the star, $z\rightarrow z^*$.
\end{enumerate}
\end{subsection}
\end{section}
\begin{section}{Symmetric powers of matrices}

\begin{subsection}{Basic construction and properties}
Given a $(d+1)\times (d+1)$ matrix, $A$, as usual, 
$(Ax)_i=\sum_j A_{ij}x_j$. For integer $N\ge 0$,
we have the matrix elements, $\bar A$, 
of the symmetric representation of $A$ in
degree $N$ defined by the relations
\begin{equation}  \label{eq:matels}
 (Ax)^m = \sum_n \bar A_{mn}x^n 
\end{equation}
 where we typically drop explicit indication of $N$ if it is understood
from the relations $N=|m|=|n|$, noting that $\bar A_{mn}=0$ unless
$|m|=|n|$. 
Monomials are ordered according to dictionary ordering with 0 ranking first, followed by $1,2,\ldots,d$. Thus the first column
of $\bar A$ gives the coefficients of $x_0^N$, etc.\\

In degree $N$, we have the trace of $\bar A$
$$\tr_{\mathrm{Sym}}^N A = \sum_m\bar A_{mm} \ .$$
For $N=0$, $\bar A_{00}=\tr_{\mathrm{Sym}}^0 A=1$, for any matrix $A$.\\

\begin{remark}\rm
If we require the degree to be indicated explicitly we write $\bar A^{(N)}$
accordingly.
\end{remark}

\begin{subsubsection}{Homomorphism property}
Successive application of $B$ then $A$ 
shows that this is a homomorphism of the multiplicative semigroup of
square $(1+d)\times (1+d)$ matrices into the multiplicative semigroup of
square $\binom{N+d}{N}\times \binom{N+d}{N}$ matrices. We call this
the \textbf{symmetric representation} of the multiplicative semigroup of matrices.
In particular, it is a representation of ${\rm GL}(d+1)$ as matrices in ${\rm GL}(\nu)$,
with $\displaystyle \nu=\binom{N+d}{N}$.\\

\begin{proposition}\label{P:hom} Matrix elements satisfy the homomorphism property
$$ \overline{AB}_{mn}=\sum_k \bar{A}_{mk} \bar{B}_{kn}\ .$$ 
\end{proposition}

\begin{proof}
  Let $y=(AB)x$ and  $w=Bx$. Then,
  \begin{eqnarray*}\label{E:mats}
    y^m &=& \sum_n \overline{AB}_{mn} \,x^n \\
        &=& \sum_k \bar{A}_{mk} \,w^k \\
        &=& \sum_n \sum_k \bar{A}_{mk}\bar{B}_{kn}\, x^n\ .
  \end{eqnarray*}
\end{proof}

\begin{remark} \rm It is important to note that the bar mapping is {\it not\ } linear, neither homogeneous nor additive.
For example, for $N>1$,
$$\overline{X+X}=\overline{2X}=2^N\bar X\ne 2\,\bar X\ .$$
\end{remark}
\end{subsubsection}
\begin{subsubsection}{Transpose Lemma}

Here we show the relation between $\overline{A^\top}$ and
${\bar A}^\top$. The idea is to consider the polynomial
bilinear form $\sum\limits_{i,j}x^iA_{ij}y^j$ and its analogous forms
in degrees $N>0$. The multinomial
expansion yields
\begin{equation}
  \label{eq:bilinear}
(\sum_{i,j}x^iA_{ij}y^j)^N=
\sum_m x^m \binom{N}{m}\,(Ay)^m =
\sum_m x^m \binom{N}{m}\,\bar A_{mn}y^n
\end{equation}
for degree $N$.\\

Replacing $A$ by $A^\top$ yields
$$(\sum_{i,j}x^iA_{ji}y^j)^N=
\sum_{m,n} x^m \binom{N}{m}\,\overline{A^\top}_{mn}y^n$$
which, interchanging indices $i$ and $j$, equals
\begin{eqnarray*}
(\sum_{i,j}y^iA_{ij}x^j)^N&=&
\sum_{m,n} y^m \binom{N}{m}\,\overline{A}_{mn}x^n
=\sum_{m,n} x^m \binom{N}{n}\,\overline{A}_{nm}y^n\\
&=&\sum_{m,n} x^m \binom{N}{n}\,{\bar A}^\top_{mn}y^n
\end{eqnarray*}
exchanging indices $m$ and $n$. 
Comparing shows that
$$\overline{A^\top}_{mn}=
\binom{N}{m}^{-1}\binom{N}{n} {\bar A}^\top_{mn}$$
In terms of $B$, we see that
\begin{lemma} {\rm Transpose Lemma}\label{prop:transpose}\\
The symmetric representation satisfies
$$\overline{A^\top} = B^{-1}{\bar A}^\top B$$
and similarly for adjoints
$$\overline{A^*} = B^{-1}{\bar A}^* B \ .$$
\end{lemma}
\end{subsubsection}

\hfill\break\\

\begin{remark}{\sl Diagonal matrices and trace}\\

Note that the $N^{\rm th}$ symmetric power of a diagonal matrix, $D$, is itself diagonal with homogeneous
monomials of the entries of the original matrix along its diagonal. In particular, the trace will be the $N^{\rm th}$ homogeneous	
symmetric function in the diagonal entries of $D$.

\begin{example}
For $V=\left(\begin{matrix}v_{0} & 0 & 0\\0 & v_{1} & 0\\0 & 0 & v_{2}\end{matrix}\right)$ we have
$$\bar V^{(2)}=
\left(\begin{matrix}v_{0}^{2} & 0 & 0 & 0 & 0 & 0\\0 & v_{0} v_{1} & 0 & 0 & 0 & 0\\0 & 0 & v_{0} v_{2} & 0 & 0 & 0\\0 & 0 & 0 & v_{1}^{2} & 0 & 0\\0 & 0 & 0 & 0 & v_{1} v_{2} & 0\\0 & 0 & 0 & 0 & 0 & v_{2}^{2}\end{matrix}\right)$$
and so on .
\end{example}	
\end{remark}

\end{subsection}
\begin{subsection}{General recurrence formulas}
Let $R$ be a $(d+1) \times (d+1)$ matrix. Let ${\bar R}^{(N-1)} \text{ and } {\bar R}^{(N)}$
denote the matrices of the symmetric tensor powers	of $R$ in degrees
$N-1$ and $N$ respectively. \\

{\normalsize Recall the usage $e_i$ for a standard multi-index corresponding
to the standard basis vector $e_i$, $0\le i \le d$.}\\

First some useful operator calculus.\\

\begin{subsubsection}{Boson operator relations}	
Let $\partial_i=\frac{\partial}{\partial x_i}$ and $\xi_i$ denote the operators of partial differentiation with respect to $x_i$ and multiplication
by $x_i$ respectively. The basic commutation relations
$$[\partial_i,\xi_j]=\delta_{ij}1$$
extend to polynomials $f(x)$ 
$$[\partial_i,f(\xi)]=\frac{\partial f}{\partial x_i}1$$
With $\partial_i1=0$, $\forall i$, we write
$$\partial_i f(\xi)1=\partial_i f(x)=\frac{\partial f}{\partial x_i}$$
evaluating as functions of $x$ as usual. \\

Now observe that if $R$ is a matrix with inverse $S$, then 
$$[\partial_\lambda S_{\lambda i},R_{j \mu}\xi_\mu]=\delta_{ij}1$$
i.e. $(\partial)S$ and $R\xi$ satisfy the boson commutation relations. Hence $(\partial)S$ will act as partial differentiation operators on
the variables $Rx$.\\

\begin{remark}For brevity, we write $(\partial)S$ as $\partial S$, where it is understood that the $\partial$ symbol stands for
the tuple $(\partial_0,\ldots,\partial_d)$ and is not acting on the matrix $S$.
\end{remark}

We will use these observations to derive recurrence relations for matrix elements of the symmetric representation.
\end{subsubsection}

\begin{proposition} 
Let $m$ be a multi-index with $|m|=N$ and $n$ be a multi-index with
$|n|=N-1$. Then we have the recurrence relation
$$\sum_{k=0}^d m_k\, {\bar R}^{(N-1)}_{m-e_k,n} R_{kj}=(n_j+1){\bar R}^{(N)}_{m,n+e_j} \ .$$
\end{proposition}
\begin{proof} 
Starting from $(Rx)^m$, use the boson commutation relations 
$$[(\partial S)_i,(R\xi)_j]=\delta_{ij}I$$
where $S=R^{-1}$. We have, evaluating as functions of $x$,
$$(\partial S)_i(Rx)^m=m_i (Rx)^{m-e_i}= \sum_n m_i\bar R_{m-e_i,n}x^n$$
and directly
$$\partial _\lambda S_{\lambda i} \sum_n \bar R_{mn} x^n=S_{\lambda i}\sum_n \bar R_{mn} n_\lambda x^{n-e_\lambda} \ .$$
Shifting the index in this last expression and comparing coefficients with the previous line yields
$$m_i\bar R_{m-e_i,n}=S_{\lambda i}\bar R_{m,n+e_\lambda} (n_\lambda+1)$$
Now  multiply both sides by $R_{ij}$ and sum over $i$ to get the result (with $k$ replacing the running index $i$).
\end{proof}
\end{subsection}

\begin{subsection}{Vector field generating the action of a matrix}
We consider the case where
$$ R=e^g$$
for $g\in \mathfrak{gl}(d+1)$. Recall the method of characteristics:\\

For a vector field $H=a(x)\cdot \partial =a(x)_\lambda \partial _\lambda$, the solution to
$$\frac{\partial u}{\partial t}=Hu$$
with initial condition $u(x,0)=f(x)$ is given by the exponential
$$u(x,t)=\exp(tH)f(x)=f(x(t))$$
where $x(t)$ satisfies the autonomous system
$$\dot x=\frac{\d x}{\d t}=a(x)$$
with initial condition $x(0)=x$. Note that $a(x)$ on the right-hand side
means $a(x(t))$ with $t$ implicit.\\

Now consider
$$H=gx\cdot \partial =g_{\mu \lambda}x_\lambda \partial _\mu$$

We have the characteristic equations
$$\dot x=gx$$
with initial condition $x(0)=x$. The solution is the matrix exponential
$$x(t)=e^{tg}x$$
which we can write as
$$x(t)=R^tx \ .$$
Hence
\begin{proposition} For $R=e^g$, 
we have the one-parameter group action
$$\exp(t\,gx\cdot \partial )f(x)=f(R^tx)  \ .$$
In particular, we have the action of $R$
$$\exp(gx\cdot \partial )f(x)=f(Rx)  \ .$$
\end{proposition}
\end{subsection}

\begin{subsection}{Gamma maps}
Let $\phi$ be any smooth homomorphism of matrix Lie groups. Then
the image
$$\phi(e^{tX})$$
of the one-parameter group $\exp(tX)$ is a one -parameter group as well. We define its
generator to be
$$\Gamma_\phi(X)$$
That is,
$$\phi(e^{tX})=e^{t\,\Gamma_\phi(X)} \ .$$
In other words,
$$\Gamma_\phi(X)=\frac{\d}{\d t}\biggm|_0 \phi(e^{tX})=\phi'(I)X$$
by the chain rule, where $\phi'$ is the jacobian of the map $\phi$. Hence
\begin{proposition}The gamma map determined by the relation
$$\phi(e^{tX})=e^{t\,\Gamma_\phi(X)}$$
can be computed as
$$\Gamma_\phi(X)=\frac{\d}{\d t}\biggm|_0 \phi(I+tX)$$
the directional derivative of $\phi$ at the identity in the direction $X$.
\end{proposition}
This provides a convenient formulation that is useful computationally as well as theoretically.\\

\begin{remark}\rm
We denote the gamma map for the symmetric representation simply by $\Gamma(X)$.\\

The properties of the gamma map as a Lie homomorphism are shown in the Appendix.
\end{remark}

The main result we will use for KG-systems follows:\\

\begin{corollary}\label{cor:gamma}\hfill\break
Given matrices $Y_1,\ldots,Y_s$. The coefficient of $v_j$ in $\displaystyle \overline{I+\sum_j v_jY_j}$ is $\Gamma(Y_j)$. That is,
$$\overline{I+\sum_j v_jY_j}=I+\sum_j v_j \Gamma(Y_j) +[\text{higher order terms in $v_j$'s}]$$
\end{corollary}
\begin{proof} Introduce the factor $t$. By linearity,
$$\overline{I+t\sum_j v_jY_j}=I+t\sum_j v_j \Gamma(Y_j)+t^2(\ldots)$$
where $t^2$ multiplies the terms higher order in the $v$'s. Setting $t=1$ yields the result.
\end{proof}

The Transpose Lemma applies to the gamma map as follows:\\

\begin{proposition} Under the $\Gamma$-map the Hermitian transpose satisfies
$$\Gamma(X^*)=B^{-1}\Gamma(X)^* B\ .$$
\end{proposition}
\begin{proof}
Write
$$\overline{X^*}=B^{-1}\bar{X}^* B$$
Exponentiating:
\begin{align*}
e^{t\Gamma(X^*)}&=\overline{e^{tX^*}}=\overline{\left(e^{tX}\right)^*}\\
&=B^{-1}\overline{e^{tX}}^* B=B^{-1}e^{t\Gamma(X)^*}B
\end{align*}
differentiating with respect to $t$ at 0 yields the result.
\end{proof}

\begin{subsubsection}{Gamma map for symmetric representation}
For the symmetric representation, take $f=x^m$ in the above Proposition. We have
\begin{align*}
\exp(t\,gx\cdot \partial ) x^m&=\sum_n (\overline{e^{tg}})_{mn} x^n\\
&=\sum_n (e^{t\Gamma(g)})_{mn} x^n \ .
\end{align*}
Now differentiate with respect to $t$ at $t=0$ to get
$$(gx\cdot \partial ) x^m=\sum_n \Gamma(g)_{mn} x^n $$
The left-hand side is
$$g_{\mu \lambda}x_\lambda \partial _\mu x^m=\sum_{j,k} g_{kj}m_k x^{m-e_k+e_j}$$
Hence
\begin{proposition} The entries of $\Gamma(g)$ are given by
$$
 \Gamma(g)_{mn}=\begin{cases}
m_i \, g_{ij}\,,& {\rm if\ }n=m-e_i+e_j \\
0,& {\rm otherwise}
\end{cases}
$$
with $i,j$ each ranging from $0$ to $d$.
\end{proposition}

\begin{remark}{\sl Diagonal matrices and trace}\\
The gamma map for a diagonal matrix will be diagonal with entries sums of choices of $N$ of the entries of the original matrix corresponding
to the monomials of the symmetric power. The trace turns out to be
$$\tr \Gamma(X)=\binom{N+d}{d+1}\,\tr X  \ .$$
\begin{example}
For $V=\left(\begin{matrix}v_{0} & 0 & 0\\0 & v_{1} & 0\\0 & 0 & v_{2}\end{matrix}\right)$ we have, for $d=2$, $N=2$,
$$\Gamma(V)=\left(\begin{matrix}2 v_{0} & 0 & 0 & 0 & 0 & 0\\0 & v_{0} + v_{1} & 0 & 0 & 0 & 0\\0 & 0 & v_{0} + v_{2} & 0 & 0 & 0\\0 & 0 & 0 & 2 v_{1} & 0 & 0\\0 & 0 & 0 & 0 & v_{1} + v_{2} & 0\\0 & 0 & 0 & 0 & 0 & 2 v_{2}\end{matrix}\right)$$
and so on .
\end{example}	
\end{remark}
\end{subsubsection}
\end{subsection}
\end{section}
\begin{section}{Columns theorem and quantum variables}
Given a matrix $A$, with each column of $A$ form a diagonal matrix. Thus, 
$$\Lambda_j=\text{diag}\,( (A_{ij}))$$
that is,
$$(\Lambda_j)_{ii}=A_{ij} \ .$$

\begin{theorem}{\rm Columns Theorem.} \label{thm:cols}\\
For any matrix $A$, let $\Lambda_j$ be the diagonal matrix formed from column $j$ of $A$. Let
$$\Lambda=\sum v_j\Lambda_j \ .$$
Then the coefficient of $v^n$ in the level $N$ induced matrix $\bar\Lambda$
is a diagonal matrix with entries the $n^{\rm th}$ column of $\bar A$.
\end{theorem}
\begin{proof}
Setting $y=\Lambda x$, we have
$$y_k=(\sum v_j A_{kj})x_k \ \Rightarrow \ y^m=(\sum {\bar A}_{mn} v^n)\,x^m \ .$$
A careful reading of the coefficients yields the result.
\end{proof}

We may express this in the following useful way:  \\

\textsl{
the diagonal entries of $\bar\Lambda$ are generating functions for the matrix elements of $\bar A$ .}

And we have
\begin{corollary}\label{cor:cols}
Let $A$ be such that the first column, label 0, consists of all 1's. Then with $\Lambda_j$ as in the above theorem, we have
$$\diag{\text{column }j \text{ of } \bar A}=\Gamma(\Lambda_j)$$
for $1\le j \le d$.
\end{corollary}
\begin{proof} In the Columns Theorem we have
$$\Lambda=v_0I+\sum_{1\le j\le d} v_j \Lambda_j \ .$$
The leading monomials are $v_0^N$, $v_0^{N-1}v_1$,\ldots,$v_0^{N-1}v_d$. These multiply diagonal matrices with entries the first $1+d$ columns of $\bar A$ respectively. Now let $v_0=1$. By Corollary \ref{cor:gamma} we have the coefficients of $v_1$ through $v_d$ given by the diagonal matrices $\Gamma(\Lambda_j)$, $1\le j\le d$. Combining these two observations yields the result.
\end{proof}

\begin{example}\label{ex:1}
Let $A=\begin{pmatrix}1&3\\2&4\end{pmatrix}$. We have
$$\Lambda_0=\begin{pmatrix}1&0\\0&2\end{pmatrix}\ ,\qquad
\Lambda_1=\begin{pmatrix}3&0\\0&4\end{pmatrix} $$
and
$$\Lambda=\begin{pmatrix}v_1+3v_2&0\\0&2v_1+4v_2\end{pmatrix}$$

We have
$$\bar A^{(2)}=\left(\begin{matrix}1 & 6 & 9\\2 & 10 & 12\\4 & 16 & 16\end{matrix}\right)$$
and
$$\bar\Lambda^{(2)}=\left(\begin{matrix}v_{0}^{2} + 6 v_{0} v_{1} + 9 v_{1}^{2} & 0 & 0\\0 & 2 v_{0}^{2} + 10 v_{0} v_{1} + 12 v_{1}^{2} & 0\\0 & 0 & 4 v_{0}^{2} + 16 v_{0} v_{1} + 16 v_{1}^{2}\end{matrix}\right)
$$
\end{example}

\begin{example}\label{ex:2}
Now consider $A=\left(\begin{matrix}1 & a & d\\1 & b & e\\1 & c & f\end{matrix}\right)$, here $a,b,c,d,e,f$ are arbitrary numbers.
 We have
$$\Lambda_0=\begin{pmatrix}1&0&0\\0&1&0\\ 0&0&1\end{pmatrix}\ ,\qquad
\Lambda_1=\left(\begin{matrix}a & 0 & 0\\0 & b & 0\\0 & 0 & c\end{matrix}\right)\ , \qquad
\Lambda_2=\left(\begin{matrix}d & 0 & 0\\0 & e & 0\\0 & 0 & f\end{matrix}\right) \ .$$
with 
$$\bar A^{(2)}=\left(\begin{matrix}1 & 2 a & 2 d & a^{2} & 2 a d & d^{2}\\1 & a + b & d + e & a b & a e + b d & d e\\1 & a + c & d + f & a c & a f + c d & d f\\1 & 2 b & 2 e & b^{2} & 2 b e & e^{2}\\1 & b + c & e + f & b c & b f + c e & e f\\1 & 2 c & 2 f & c^{2} & 2 c f & f^{2}\end{matrix}\right)$$
and
$$
\Gamma(\Lambda_1)=\left(\begin{matrix}2 a & 0 & 0 & 0 & 0 & 0\\0 & a + b & 0 & 0 & 0 & 0\\0 & 0 & a + c & 0 & 0 & 0\\0 & 0 & 0 & 2 b & 0 & 0\\0 & 0 & 0 & 0 & b + c & 0\\0 & 0 & 0 & 0 & 0 & 2 c\end{matrix}\right),\,
\Gamma(\Lambda_2)=\left(\begin{matrix}2 d & 0 & 0 & 0 & 0 & 0\\0 & d + e & 0 & 0 & 0 & 0\\0 & 0 & d + f & 0 & 0 & 0\\0 & 0 & 0 & 2 e & 0 & 0\\0 & 0 & 0 & 0 & e + f & 0\\0 & 0 & 0 & 0 & 0 & 2 f\end{matrix}\right)
$$
accordingly.
\end{example}
\end{section}
\begin{section}{Associated and quantum variables}\label{sec:qv}
Recall that in the finite-dimensional case, quantum observables are Hermitian matrices, thus unitarily
equivalent to a real diagonal matrix.\\

Let $A$ satisfy the condition that its first column consists of all 1's.
Form $\Lambda_j$ from $A$ as in the Columns Theorem. Define, for each $j$,
$$X_j=A^{-1}\Lambda_j A \ .$$

For $X_j$ so defined, we have the following terminology.

\begin{definition} Given $A$. Define an \textsl{associated variable} to be a matrix equivalent to a diagonal matrix with $A$ as similarity transformation,
where the entries of the diagonal matrix are from a column of $A$.\\

If it turns out that this yields a unitary equivalence, for complex $A$,
we call them \textsl{quantum variables}. For real $A$, they will be quantum observables.
\end{definition}

\begin{remark}Note that our standard choice of $A$ will have $\Lambda_0$ always equal to the identity matrix and hence the same for $X_0$.
\end{remark}

Set $X=\sum v_jX_j$. Then with $\Lambda$ as in the Columns Theorem, we have the spectral relations
\begin{align*} AX&=\Lambda A\\
A e^{tX}&= e^{t \Lambda}A
\end{align*}
barring and taking derivatives at 0 yields $\bar A \Gamma(X)= \Gamma(\Lambda) \bar A$. Taking adjoints we have
\begin{equation}\label{eq:recspec}
\Gamma(X)^*(\bar A)^*=  (\bar A)^*\Gamma(\Lambda)^*
\end{equation}
with columns of $(\bar A)^*$ eigenvectors of $\Gamma(X)^*$, i.e., this is a spectral resolution of $\Gamma(X)^*$. By Corollary \ref{cor:cols}, the corresponding eigenvalues are complex conjugates of the entries of the columns of $\bar A$. For real $A$ we have
real spectra and hence quantum observables.\\

\begin{example}\label{ex:3}
Take $A=\begin{pmatrix}1&3\\1&4\end{pmatrix}$. We have
$$\Lambda_0=\begin{pmatrix}1&0\\0&1\end{pmatrix}\ ,\qquad
\Lambda_1=\begin{pmatrix}3&0\\0&4\end{pmatrix} \ . $$
With $X_0$ the identity, we have
$$X_1= \left(\begin{matrix}0 & -12\\1 & 7\end{matrix}\right)
\qquad \hbox{and}\qquad \Gamma(X_1)=\left(\begin{matrix}0 & -24 & 0\\1 & 7 & -12\\0 & 2 & 14\end{matrix}\right) \ .$$

We find $\bar A=\left(\begin{matrix}1 & 6 & 9\\1 & 7 & 12\\1 & 8 & 16\end{matrix}\right)$ and the matrix relations

$$\left(\begin{matrix}0 & 1 & 0\\-24 & 7 & 2\\0 & -12 & 14\end{matrix}\right)
\left(\begin{matrix}1 & 1 & 1\\6 & 7 & 8\\9 & 12 & 16\end{matrix}\right)
=\left(\begin{matrix}1 & 1 & 1\\6 & 7 & 8\\9 & 12 & 16\end{matrix}\right)
\left(\begin{matrix}6 & 0 & 0\\0 & 7 & 0\\0 & 0 & 8\end{matrix}\right) \ .$$
\end{example}
\end{section}

\begin{section}{K-condition and associated system}
Start with $U$, a unitary matrix with $|U_{i0}|>0$, $0\le i\le d$. \\[.1in]
Make the first column to consist of all 1's as follows. Let 
$$\delta=\begin{pmatrix} U_{00}&& {}\\&\ddots &\\ && U_{d0}\end{pmatrix}$$
the diagonal matrix with diagonal entries the first column of $U$.
$D$ can be any diagonal matrix with positive diagonal entries and $D_{00}=1$.
The initial probabilities are given by the diagonal matrix
$$\p=\delta^*\delta=\begin{pmatrix} |U_{00}|^2&& {}\\&\ddots &\\ && |U_{d0}|^2\end{pmatrix}=
\begin{pmatrix}  p_0&& {}\\&\ddots &\\ && p_d\end{pmatrix}$$
with $p_i>0$, $\tr \p=1$.

\begin{remark}\rm We consider the values $p_i$ as probabilities corresponding to a multinomial process. Namely,
a succession of independent trials each having one of the same $d$ possible outcomes, with
$p_0$ the probability of none of those outcomes occurring and $p_i$ the probability of outcome $i$,
$1\le i \le d$.

\end{remark}

The \textsl{generating matrix} $A$ is defined by
$$A=\delta^{-1}\,U\,\sqrt{D}\ .$$

The essential property satisfied by $A$ is
\begin{definition} \label{def:kcond}\rm    We say that $A$ satisfies the \textsl{$K$-condition} if
there exists a positive diagonal probability matrix $\p$ and a positive diagonal matrix $D$ such that
$$ A^* \p A=D \ .$$
with $A_{i0}=1$, $0\le i\le d$.
\end{definition}

\begin{remark}
1. Note that this says that the columns of $A$ are orthogonal with respect to the weights $p_i$ and squared norms given by $D_{ii}$.\\

2. \textit{A fortiori} $A^{-1}=D^{-1}A^*\p$.
\end{remark}

Let's verify the condition for our construction.
\begin{proposition}
$A=\delta^{-1}\,U\,\sqrt{D}$, as defined above, satisfies the $K$-condition.
\begin{proof}We have
\begin{align*}
A^*\p A&=\sqrt{D}U^*(\delta^{-1})^*\delta^*\delta \delta^{-1}U\sqrt{D}\\
&=\sqrt{D}U^*U\sqrt{D}=D
\end{align*}
as required.
\end{proof}
\end{proposition}

Note that the $K$-condition implies that
$$U=\sqrt{\p}\,A\frac{1}{\sqrt{D}}$$
is unitary.\\

\begin{paragraph}{Some contexts}
\begin{enumerate}
\item \textsl{Gaussian quadrature\,.}
Let $\{\phi_0,\ldots,\phi_n\}$ be an orthogonal polynomial sequence with positive weight function on a finite interval $\I$ of the real line.
For Gaussian quadrature,
$$\frac{1}{|\I|}\,\int_{\I} f \approx \sum_k w_k f(x_k)$$
with $x_k$ the zeros of  $\phi_n$ and appropriate weights $w_k$.  Let 
$$A_{ij}=\phi_{i-1}(x_j)$$
Then, with $\Gamma$ the diagonal matrix of squared norms, $\Gamma_{ii}=\|\phi_i\|^2$, we have
$$AWA^{*}=\Gamma $$
where $W$ is the diagonal matrix with $W_{kk}=w_k$. That is, $A^*$ satisfies the $K$-condition.\\

\item \textsl{Association schemes\,.}
Given an association scheme with adjacency matrices $A_i$, the $P$ and $Q$ matrices correspond to the decomposition of 
the algebra generated by the $A_i$ into an orthogonal direct sum, the entries $P_{ij}$ being the corresponding eigenvalues.
A basic result is the relation
$$P^{*} D_\mu P=v\,D_v$$
where $D_\mu$ is the diagonal matrix of multiplicities and $D_v$ the diagonal matrix of valencies of the scheme.
This plays an essential r\^ole in the work of Delsarte, Bannai, and generally in this area.\\

\end{enumerate}
\end{paragraph}
\begin{subsection}{Matrices for multivariate Krawtchouk systems}
For any degree $N$, the $K$-condition implies
$$ \overline{A^*}\bar \p\bar A=\bar D \ .$$
Application of  the {  Transpose Lemma}
$$B\overline{A^*}=\bar A^* B$$ 
with $B$ the special multinomial diagonal matrix yields
$$\bar A^*\, B\bar \p\, \bar A = B\bar D \ .$$
\begin{definition} The \textsl{Krawtchouk matrix} $\Phi$ is defined as
$$\Phi=\bar A^*$$
where $A$ is a matrix satisfying the $K$-condition.
\end{definition}

Thus,
\begin{proposition} \label{prop:kdef}
For $A$ satisfying the $K$-condition, the Krawtchouk matrix $\Phi=\bar A^*$ satisfies the orthgonality relation
$$\Phi\, B\bar \p\, \Phi^*=B\bar D \ .$$
\end{proposition}

The entries of $\Phi$ are the values of the {\bf   multivariate Krawtchouk polynomials} 
thus determined. \\

The rows of $\Phi$ are comprised of the values of the corresponding Krawtchouk polynomials. The relationship above
indicates that these are orthogonal with respect to the associated multinomial distribution.
$B\bar D$ is the diagonal matrix of squared norms according to the orthogonality of the Krawtchouk polynomials.\\

Note that the unitarity of $\displaystyle U=\sqrt{\p}\,A\frac{1}{\sqrt{D}}$ entails
$UU^*=I$ as well. We have
\begin{proposition}The Krawtchouk matrix $\Phi$ satisfies dual orthogonality relations
$$\Phi^* (B\bar D)^{-1}\Phi= (B\bar\p)^{-1} \ .$$
\end{proposition}
\begin{proof} Rewrite the relation in Proposition \ref{prop:kdef} as
$$\Phi\, B\bar \p\, \Phi^*(B\bar D)^{-1}=I \ .$$
This says $\Phi\, B\bar \p$ and $\Phi^*(B\bar D)^{-1}$ are mutual inverses. Reversing the order gives the dual form.
\end{proof}

\begin{example}\rm
Start with the orthogonal matrix
$\displaystyle U= \begin{pmatrix}\sqrt{q}&\sqrt{p}\\ \sqrt{p}&-\sqrt{q}\end{pmatrix}\ .$ \\
Factoring out the squares from the first column we have 
$$\p={ \begin{pmatrix}q&0\\ 0&p\end{pmatrix}} \quad\text{and}\quad
A= \begin{pmatrix}1&p\\ 1&-q\end{pmatrix}\ .$$
with
$$A^*\p A={\begin{pmatrix}1&0\\ 0& pq\end{pmatrix}}=D\ .$$ 
Take $N=4$. \\
We have the { Krawtchouk matrix} 
$\Phi={\bar A}^*=$
$$ \begin{pmatrix} 1&1&1&1&1\\4p&-q+3p&-2q+2p&-3q+p&-4q\\
6{p}^{2}&-3pq+3{p}^{2}&{q}^{2}-4pq+{p}^{2}&3{q}^{2}-3pq&6{q}^{2}\\
4{p}^{3}&-3{p}^{2}q+{p}^{3}&2p{q}^{2}-2{p}^{2}q&-{q}^{3}+3p{q}^{2}&-4{q}^{3}\\
{p}^{4}&-{p}^{3}q&{p}^{2}{q}^{2}&-p{q}^{3}&{q}^{4}\end{pmatrix} \ .$$
$\p$ becomes the induced matrix  
$$\bar \p= \begin{pmatrix} q^4&0&0&0&0\\0&q^3p&0&0&0\\0&0&q^2p^2&0&0\\
0&0&0&qp^3&0\\0&0&0&0&p^4\end{pmatrix}\ .$$  \\
and the binomial coefficient matrix $B=\text{diag}(1,4,6,4,1)$.
\end{example}

\begin{example}\label{ex:two}
For an example in two variables, start with 
$$U=\left( \begin {array}{ccc} 1/\sqrt {3}&1/\sqrt {3}&1/\sqrt {3}\\\noalign{\medskip}1/\sqrt {2}&-1/\sqrt {2}&0
\\\noalign{\medskip}1/\sqrt {6}&1/\sqrt {6}&-2/\sqrt {6}
\end {array} \right) \ .$$
Taking $D$ to be the identity, we have
$$A= \left( \begin {array}{rrr} 1&1&1\\\noalign{\medskip}1&-1&0
\\\noalign{\medskip}1&1&-2\end {array} \right) \,,\qquad
\p= \left( \begin {array}{ccc} 1/3&0&0\\\noalign{\medskip}0&1/2&0
\\\noalign{\medskip}0&0&1/6\end {array} \right) \ . $$
The level two induced matrix
$$\Phi^{(2)}=\left( \begin {array}{rrrrrr} 1&1&1&1&1&1\\\noalign{\medskip}2&0&2&-2
&0&2\\\noalign{\medskip}2&1&-1&0&-2&-4\\\noalign{\medskip}1&-1&1&1&-1&
1\\\noalign{\medskip}2&-1&-1&0&2&-4\\\noalign{\medskip}1&0&-2&0&0&4
\end {array} \right)$$
indicating the level explicitly. 
\end{example}

\begin{example}\rm
An example for three variables.\\

Here we illustrate a special property of reflections. Start with the vector $\v^T=(1,-1,-1,-1)$. 
Form the rank-one projection, $V=\v\v^T/\v^T\v$, and  the corresponding reflection $U=2V-I$:
$$U=\left( \begin {array}{cccc} -1/2&-1/2&-1/2&-1/2\\\noalign{\medskip}-1/2&-1/2&1/2&1/2\\\noalign{\medskip}-1/2&1/2&-1/2&1/2
\\\noalign{\medskip}-1/2&1/2&1/2&-1/2\end {array} \right) \ . $$

Rescale so that the first column consists of all 1's to get
$$A= \left( \begin {array}{rrrr} 1&1&1&1\\\noalign{\medskip}1&1&-1&-1
\\\noalign{\medskip}1&-1&1&-1\\\noalign{\medskip}1&-1&-1&1\end {array} \right)$$
with the uniform distribution $p_i=1/4$, and $D=I$. We find
$$\Phi^{(2)}= \left( \begin {array}{rrrrrrrrrr} 1&1&1&1&1&1&1&1&1&1
\\\noalign{\medskip}2&2&0&0&2&0&0&-2&-2&-2\\\noalign{\medskip}2&0&2&0&
-2&0&-2&2&0&-2\\\noalign{\medskip}2&0&0&2&-2&-2&0&-2&0&2
\\\noalign{\medskip}1&1&-1&-1&1&-1&-1&1&1&1\\\noalign{\medskip}2&0&0&-
2&-2&2&0&-2&0&2\\\noalign{\medskip}2&0&-2&0&-2&0&2&2&0&-2
\\\noalign{\medskip}1&-1&1&-1&1&-1&1&1&-1&1\\\noalign{\medskip}2&-2&0&0
&2&0&0&-2&2&-2\\\noalign{\medskip}1&-1&-1&1&1&1&-1&1&-1&1\end {array} \right)$$
With $A^2=4I$, we have, in addition to the orthogonality relation,  that $(\Phi^{(2)})^2=16I$.
\end{example}
\end{subsection}

\begin{subsection}{Quantum variables and recurrence formulas}
We resume from \S\ref{sec:qv}.
Let $A$ satisfy the $K$-condition.
Form $\Lambda_j$ from $A$ as in the Columns Theorem. Define, for each $j$,
$$X_j=A^{-1}\Lambda_j A \ .$$
Observe that from the $K$-condition we have
\begin{align*}
A^{-1}\Lambda_jA&=D^{-1}A^*\delta^*\delta \Lambda_j\delta^{-1}U\sqrt{D}\\
&=D^{-1}\sqrt{D}U^*(\delta^{-1})^*\delta^*\delta\Lambda_j\delta^{-1}U\sqrt{D}\\
&=\frac{1}{\sqrt{D}}\,U^*\Lambda_jU\sqrt{D} \ .
\end{align*}

Thus,
\begin{proposition} If $D$ equals the identity in the $K$-condition for $A$, we have
$$X_j=A^{-1}\Lambda_j A=U^*\Lambda_j U$$
that is, conjugation by $A$ is unitary equivalence. Thus, the $X_j$ are quantum observables in the real case,
and quantum variables generally.
\end{proposition}
Referring to equation \eqref{eq:recspec}, we have, for each $j$, $0\le j\le d$,
$$\Gamma(X_j)^*\Phi=  \Phi\Gamma(\Lambda_j)^* \ .$$

The left hand side induces combinations of the rows of $\Phi$ while the right hand side multiplies by diagonal elements of $\Gamma(\Lambda_j)^*$.
Now Corollary \ref{cor:cols} tells us that these are (complex conjugates) of the columns from $\bar A$, equivalently, rows of $\Phi$. These
are the rows with labels $1$ through $d$ thus corresponding to the polynomials $x_j$, $1\le j\le d$. Thus, we interpret the above equation
as the relation
$$(\text{Rec})\, \Phi=\Phi\, (\text{Spec})$$
that is, these are recurrence relations for the multivariate Krawtchouk polynomials.\\

In this next example, we keep the variable $v$ and see that our procedure gives multiplication by higher-order polynomials
thus yielding more complex recurrence formulas.

\begin{example}\label{ex:binom}
Start with the orthogonal matrix
$\displaystyle U=\begin{pmatrix}1/\sqrt{2}&1/\sqrt{2}\\ 1/\sqrt{2}&-1/\sqrt{2}\end{pmatrix}\ .$\\[.1in]
Factoring out the squares from the first column we have 
$$\p={\begin{pmatrix}1/2&0\\ 0&1/2\end{pmatrix}} \quad\text{and}\quad
A=\begin{pmatrix}1&1\\ 1&-1\end{pmatrix}\ .$$
satisfying $A^* \p A=I$.
We have the Krawtchouk matrix in degree 4
$$\Phi={\bar A}^*=\begin{pmatrix} 1&1&1&1&1\\ 4&2&0&-2&-4\\ 6&0&-2&0&6\\ 4&-2&0&2&-4\\ 1&-1&1&-1&1\end{pmatrix} \ .$$
The entries of $\p$ become $\displaystyle \bar \p=\frac{1}{16}\,I_5$ 
with the binomial coefficient matrix $B=\text{diag}(1,4,6,4,1)$.\\

Take the second column of $A$, recalling that our indexing starts with 0,
 and form the diagonal matrix $$\Lambda_1=\begin{pmatrix}1&0\\ 0&-1\end{pmatrix}\ .$$
The corresponding observable is $$X_1= A^{-1}\Lambda_1A=\begin{pmatrix}0&1\\ 1&0\end{pmatrix}\, .$$
 Let $\displaystyle \Lambda={ I+v\Lambda_1}\quad\text{and}\quad X= I+v\,X_1$.\\
Then in degree 4, $\bar\Lambda=$
$$\diag{(1+v)^4,(1+v)^3(1-v),(1+v)^2(1-v)^2,(1+v)(1-v)^3,(1-v)^4} \ .$$
And $\bar X=$
$$\begin{pmatrix} 1&4v&6v^2&4v^3&v^4\\ v&1+3v^2&3v+3v^3&3v^2+v^4&v^3\\ 
                                 v^2&2v+2v^3&1+4v^2+v^4&2v+2v^3&v^2\\ v^3&3v^2+v^4&3v+3v^3&1+3v^2&v\\ v^4&4v^3&6v^2&4v&1\end{pmatrix} \ .$$

Now we have the spectrum via the coefficient of $v$ in $\bar\Lambda$
$$\text{Spec}=\text{diag}(4,2,0,-2,-4)$$
which is the same as the row with label 1 in $\Phi$.
The coefficient of $v$ in the transpose of $\bar X$ give the recurrence coefficients
$$\text{Rec}=\begin{pmatrix} 0&1&0&0&0\\ 4&0&2&0&0\\ 0&3&0&3&0\\ 0&0&2&0&4\\ 0&0&0&1&0\end{pmatrix}$$
satisfying the relation
$$(\text{Rec})\, \Phi=\Phi\, (\text{Spec})$$
which is essentially the recurrence relation for the corresponding Krawtchouk polynomials.
Coefficients of higher powers of $v$ thus correspond to higher-order recurrence relations corresponding
to multiplication by higher order Krawtchouk polynomials. For example, the coefficient of $v^2$ yields the relations
\begin{align*}
\left( \begin {array}{ccccc} 0&0&1&0&0\\\noalign{\medskip}0&3&0&3&0\\\noalign{\medskip}6&0&4&0&6\\\noalign{\medskip}0&3&0&3&0
\\\noalign{\medskip}0&0&1&0&0\end {array} \right)
&\begin{pmatrix} 1&1&1&1&1\\ 4&2&0&-2&-4\\ 6&0&-2&0&6\\ 4&-2&0&2&-4\\ 1&-1&1&-1&1\end{pmatrix}\\
&=\begin{pmatrix} 1&1&1&1&1\\ 4&2&0&-2&-4\\ 6&0&-2&0&6\\ 4&-2&0&2&-4\\ 1&-1&1&-1&1\end{pmatrix}
 \left(\begin {array}{ccccc} 6&0&0&0&0\\\noalign{\medskip}0&0&0&0&0
\\\noalign{\medskip}0&0&-2&0&0\\\noalign{\medskip}0&0&0&0&0
\\\noalign{\medskip}0&0&0&0&6\end {array} \right)
\end{align*}
etc.
\end{example}
\begin{example}\rm
Continuing with Example \ref{ex:two},
$$A= \left( \begin {array}{rrr} 1&1&1\\\noalign{\medskip}1&-1&0
\\\noalign{\medskip}1&1&-2\end {array} \right) \,,\qquad
\p= \left( \begin {array}{ccc} 1/3&0&0\\\noalign{\medskip}0&1/2&0
\\\noalign{\medskip}0&0&1/6\end {array} \right) $$
and $D$ the identity. The level two Krawtchouk matrix 
$$\Phi=\left( \begin {array}{rrrrrr} 1&1&1&1&1&1\\\noalign{\medskip}2&0&2&-2
&0&2\\\noalign{\medskip}2&1&-1&0&-2&-4\\\noalign{\medskip}1&-1&1&1&-1&
1\\\noalign{\medskip}2&-1&-1&0&2&-4\\\noalign{\medskip}1&0&-2&0&0&4
\end {array} \right) \ .$$
We have the (pre)spectral matrices
$$\Lambda_1=\diag{1,-1,1} \qquad \hbox{and} \qquad \Lambda_2=\diag{1,0,-2}$$
with $\Lambda_0$ the $3\times 3$ identity. Conjugating by $A$ yields
$$X_1=    \left(\begin{matrix}0 & 1 & 0\\1 & 0 & 0\\0 & 0 & 1\end{matrix}\right)
\qquad \hbox{and} \qquad X_2=\left(\begin{matrix}0 & 0 & 1\\0 & 0 & 1\\1 & 1 & -1\end{matrix}\right) \ .$$
Cranking up to degree 2 we have
$$\Gamma(X_1)^*= \left(\begin{matrix}0 & 1 & 0 & 0 & 0 & 0\\2 & 0 & 0 & 2 & 0 & 0\\0 & 0 & 1 & 0 & 1 & 0\\0 & 1 & 0 & 0 & 0 & 0\\0 & 0 & 1 & 0 & 1 & 0\\0 & 0 & 0 & 0 & 0 & 2\end{matrix}\right)
\qquad \hbox{and} \qquad \Gamma(X_2)^*=\left(\begin{matrix}0 & 0 & 1 & 0 & 0 & 0\\0 & 0 & 1 & 0 & 1 & 0\\2 & 1 & -1 & 0 & 0 & 2\\0 & 0 & 0 & 0 & 1 & 0\\0 & 1 & 0 & 2 & -1 & 2\\0 & 0 & 1 & 0 & 1 & -2\end{matrix}\right)$$
with spectra
$$\Gamma(\Lambda_1)= \left(\begin{matrix}2 & 0 & 0 & 0 & 0 & 0\\0 & 0 & 0 & 0 & 0 & 0\\0 & 0 & 2 & 0 & 0 & 0\\0 & 0 & 0 & -2 & 0 & 0\\0 & 0 & 0 & 0 & 0 & 0\\0 & 0 & 0 & 0 & 0 & 2\end{matrix}\right)
\qquad \hbox{and} \qquad \Gamma(\Lambda_2)=\left(\begin{matrix}2 & 0 & 0 & 0 & 0 & 0\\0 & 1 & 0 & 0 & 0 & 0\\0 & 0 & -1 & 0 & 0 & 0\\0 & 0 & 0 & 0 & 0 & 0\\0 & 0 & 0 & 0 & -2 & 0\\0 & 0 & 0 & 0 & 0 & -4\end{matrix}\right)$$
with no adjoints necessary in the real case.
\end{example}
\end{subsection}
\end{section}

\begin{section}{Illustrative cases}

\begin{subsection}{Reflections}\label{sec:ref}

Given a vector $\v$, form the reflection $$U=R_{\v}=2(\v\v^*/\v^*\v)-I$$ 
fixing $\v$.
This is both self-adjoint and unitary, $U^2=I$. The Krawtchouk	matrix $\Phi$ will
have interesting special properties in this case.

\begin{proposition} Let $U$ be a unitary reflection and let $A$ be the corresponding
generating matrix for the KG-system:
$$A=\delta^{-1}U\sqrt{D}$$
where $\delta$ is the diagonal matrix with diagonal elements from the first column of $U$.
Let $\Phi$ be the Krawtchouk matrix for a given degree. Then:\\

1. {\rm Involutive property:} $\left( \overline{\delta^* D^{-1/2}}\,\Phi\right)^2=I$ .\\

2. {\rm Self-adjointness:} $\Phi B \bar\delta^*\overline{D^{1/2}}$ is self-adjoint.
\end{proposition}
\begin{proof} For \#1,
$$U^2=\delta A D^{-1/2}\delta A D^{-1/2}=I$$
Now multiply by $\delta^{-1}$ on the left and multiply back by $\delta$ on the right to get
$$(A \delta D^{-1/2})^2=I \ .$$
Now bar and star to get
$$(\overline{ \delta^* D^{-1/2}} \bar A^*)^2=I $$
and replacing $\bar A^*$ with $\Phi$ yields the result.\\

For \#2, start with
$$\delta A D^{-1/2}=D^{-1/2}A^*\delta^* \ .$$
Rearrange to
$$\delta D^{1/2} A= A^* \delta^* D^{1/2} \ .$$
Now bar to get
$$\bar\delta\, \overline{D^{1/2}} \bar A= \overline{A^*} \bar\delta^* \overline{D^{1/2}} \ .$$
In terms of $\Phi$, this is
$$\bar\delta\, \overline{D^{1/2}} \Phi^*= B^{-1}\Phi B \bar\delta^* \overline{D^{1/2}} \ .$$
Multiplying through by $B$ on the left yields
$$\bar\delta\, \overline{D^{1/2}} B\Phi^*= \Phi B \bar\delta^* \overline{D^{1/2}} $$
as required.
\end{proof}
\begin{example}Here's an example which illustrates the various quantities involved.
Starting with $\v=(1,2 i)^T$, form the reflection 
$$U=R_{\v}=\left(\begin{matrix}- \frac{3}{5} & - \frac{4 i}{5} \vstrut\\
\frac{4 i}{5} & \frac{3}{5}\end{matrix}\right) \ .$$
Factoring out $\delta=\left(\begin{matrix}- \frac{3}{5} & 0\\0 & \frac{4 i}{5}\end{matrix}\right)$, we take
$D=\left(\begin{matrix}1 & 0\\0 & 36\end{matrix}\right)$
so that
$$A=\delta^{-1}U\sqrt{D}=\left(\begin{matrix}1 & 8 i\\1 & - \frac{9 i}{2}\end{matrix}\right)\qquad\hbox{and}\qquad
 \p=\left(\begin{matrix}\frac{9}{25} & 0\\0 & \frac{16}{25}\end{matrix}\right) \ .$$
In degree 3, 
$$\Phi=\left(\begin{matrix}1 & 1 & 1 & 1\\- 24 i & - \frac{23 i}{2} & i & \frac{27 i}{2}\\-192 & 8 & \frac{207}{4} & - \frac{243}{4} \vstrut\\
512 i & - 288 i & 162 i & - \frac{729 i}{8} \vstrut\end{matrix}\right)$$
with $B=\diag{1,3,3,1}$.\\

We have the involution
$$ \overline{\delta^* D^{-1/2}}\,\Phi=\left(\begin{matrix}- \frac{27}{125} & - \frac{27}{125} & - \frac{27}{125} & - \frac{27}{125} \vstrut\\
- \frac{144}{125} & - \frac{69}{125} & \frac{6}{125} & \frac{81}{125} \vstrut\\
- \frac{256}{125} & \frac{32}{375} & \frac{69}{125} & - \frac{81}{125} \vstrut\\
- \frac{4096}{3375} & \frac{256}{375} & - \frac{48}{125} & \frac{27}{125} \vstrut\end{matrix}\right)$$
and the self-adjoint matrix
$$\Phi B \bar\delta^*\overline{D^{1/2}}=\left(\begin{matrix}- \frac{27}{125} & - \frac{648 i}{125} & \frac{5184}{125} & \frac{13824 i}{125} \vstrut\\
\frac{648 i}{125} & - \frac{7452}{125} & \frac{5184 i}{125} & - \frac{186624}{125} \vstrut\\
\frac{5184}{125} & - \frac{5184 i}{125} & \frac{268272}{125} & - \frac{839808 i}{125} \vstrut\\
- \frac{13824 i}{125} & - \frac{186624}{125} & \frac{839808 i}{125} & \frac{1259712}{125} \vstrut\end{matrix}\right) \ .$$
\end{example}
\end{subsection}

\begin{subsection}{Symmetric binomial}
Let us look at the binomial case with equal probabilities. This is the basic  case
for the original version of Krawtchouk's polynomials.\\

We start with the orthogonal matrix
$\displaystyle U=\begin{pmatrix}1/\sqrt{2}&1/\sqrt{2}\\ 1/\sqrt{2}&-1/\sqrt{2}\end{pmatrix}$
with
$$\p={\begin{pmatrix}1/2&0\\ 0&1/2\end{pmatrix}} \quad\text{and}\quad
A=\begin{pmatrix}1&1\\ 1&-1\end{pmatrix}\ .$$
satisfying $A^* \p A=I$.\\

Recalling Example \ref{ex:binom}, consider $N=4$ from which we will see the pattern for general $N$. We have the recurrence relations
\begin{align*}
\begin{pmatrix} 0&1&0&0&0\\ 4&0&2&0&0\\ 0&3&0&3&0\\ 0&0&2&0&4\\ 0&0&0&1&0\end{pmatrix}
&\begin{pmatrix} 1&1&1&1&1\\ 4&2&0&-2&-4\\ 6&0&-2&0&6\\ 4&-2&0&2&-4\\ 1&-1&1&-1&1\end{pmatrix}\\
&=
\begin{pmatrix} 1&1&1&1&1\\ 4&2&0&-2&-4\\ 6&0&-2&0&6\\ 4&-2&0&2&-4\\ 1&-1&1&-1&1\end{pmatrix}
\begin{pmatrix} 4&0&0&0&0\\ 0&2&0&0&0\\ 0&0&0&0&0\\ 0&0&0&-2&0\\ 0&0&0&0&-4\end{pmatrix} \ .
\end{align*}

With $j$ the column index, the spectrum has the form $x=N-2j$. Denoting by $K_n$ the rows of the Krawtchouk	 matrix, i.e.,
these are the Krawtchouk polynomials when expressed as functions of $x$, we read off the recurrence relations
$$(N-(n-1))K_{n-1}+(n+1)K_{n+1}=xK_n$$
with initial conditions $K_0=1$, $K_1=x$. The polynomials may be expressed in terms of either $x$ or $j$. For example
\begin{align*}
K_1=x \ ,&&&K_1=N-2j\\
K_2=\frac12(x^2-N)\ ,&&&K_2=\frac12(4j^2-4j+N^2-N)\\
K_3=\frac16(x^3-3xN-2x)\ ,&&& \hbox{\ldots}
\end{align*}

Recalling $\bar\Lambda$ from Example \ref{ex:binom} we see the form for the generating function
$$(1+v)^{N-j}(1-v)^j=\sum_n v^n K_n(j)$$
via the Columns Theorem.\\

In this context $U$ is a reflection, with $p=(1/2)I$, $D=I$. We have, specializing from \S \ref{sec:ref},
$$\Phi^2=2^N\,I \qquad\hbox{and}\qquad (\Phi B)^*= \Phi B$$
involutive and symmetry properties respectively.
\end{subsection}

\end{section}

\begin{section}{Conclusion}
Here we have presented the approach to multivariate Krawtchouk polynomials using matrices, working directly
with the values they take. With this approach the algebraic structures are clearly seen.\\

In KG-Systems II, we present the analytical side. In particular, the KG-systems as
Bernoulli systems is presented showing how these are discrete quantum systems. The Fock space structure
and observables are discussed in detail.\\

We conclude with the observation that there are many unanswered questions involving Krawtchouk matrices, for example,
the behavior of the largest eigenvalue for the symmetric Krawtchouk matrices in the classical case is an interesting
open problem \cite{FF}.\\

\begin{remark} Symbolic computations have been done using IPython.
Some IPython notebooks for computations and examples are available at:
\begin{quote}\texttt{ http://ziglilucien.github.io/NOTEBOOKS }\end{quote}
which will be updated as that project develops.
\end{remark}

Finally, as these are finite quantum systems, possible connections with quantum computation are intriguing.
\end{section}

\begin{section}{Appendix}
\begin{theorem} Basic properties of the $\Gamma$-map for the symmetric representation.\\
\\
The first two properties show that $\Gamma$ is a linear map. The third shows that it is a Lie homomorphism.\\

1. For scalar $\lambda$, $\Gamma(\lambda X)=\lambda \Gamma(X)$.\\

2. Additivity holds: $\Gamma(X+Y)=\Gamma(X)+\Gamma(Y)$.\\

3. With $[X,Y]=XY-YX$ denoting the commutator of $X$ and $Y$, we have
$$\Gamma([X,Y])=[\Gamma(X),\Gamma(Y)]\ .$$
\end{theorem}
\begin{proof} For Property \#1 start from the definition:
$$\overline{e^{t\lambda X}}=e^{\lambda t\Gamma(X)}=e^{t\Gamma(\lambda X)}$$
and differentiating with respect to $t$ at 0 gives \#1.\\

For \#2, 
$$\overline{e^{t(X+Y)}}=e^{t\Gamma(X+Y)}$$
by definition. Now we use multiplicativity and the Trotter product formula as follows:
\begin{align*}
\overline{e^{t(X+Y)}}&=\lim_{n\to\infty}\overline{\left(e^{(t/n)X}e^{(t/n)Y}\right)^n}\\
&=\lim_{n\to\infty}\left(\overline{e^{(t/n)X}}\ \overline{e^{(t/n)Y}}\right)^n\\
&=\lim_{n\to\infty}\left(e^{(t/n)\Gamma(X)}\ e^{(t/n)\Gamma(Y)}\right)^n\\
&=e^{t(\Gamma(X)+\Gamma(Y))}
\end{align*}
and differentiating with respect to $t$ at 0 yields the result.\\

For \#3 we use  the adjoint representation, with $({\rm ad\,}X)Y=[X,Y]$,
$$e^{tX}Ye^{-tX}=e^{t \,{\rm ad\,}X} Y=Y+t[X,Y]+\frac{t^2}{2}\,[X,[X,Y]]+\cdots+\frac{t^n}{n!}({\rm ad\,}X)^nY+\cdots$$
Write this as
$$e^{tX}Ye^{-tX}=Y+t[X,Y]+\mathcal{O}(t^2)$$
Then
$$\overline{e^{tX}e^{sY}e^{-tX}}=\overline{e^{s(Y+t[X,Y]+\mathcal{O}(t^2))}}$$
By multiplicativity	and linearity we have
$$e^{t\Gamma(X)}e^{s\Gamma(Y)}e^{-t\Gamma(X)}=e^{s\Gamma(Y)+st\Gamma([X,Y])+s\mathcal{O}(t^2)}$$
Now, differentiating with respect to $s$ at 0 yields
$$e^{t\Gamma(X)}\Gamma(Y)e^{-t\Gamma(X)}=\Gamma(Y)+t\Gamma([X,Y])+\mathcal{O}(t^2)$$
Differentiating with respect to $t$ at 0 finishes the proof.
\end{proof}
\end{section}

\begin{paragraph}{\bf Acknowledgment.}
The author would like to thank the organizers of the $33^{\rm nd}$ Quantum Probability conference, Luminy, 2012, where
a portion of this work was presented.
\end{paragraph}

\end{document}